\documentclass[journal]{IEEEtran}  
\ifCLASSINFOpdf

\else

\fi

\usepackage{epsfig} 
\usepackage{amsmath} 
\usepackage{subfigure}
\usepackage{color}

\newtheorem{lemma}{Lemma}
\newtheorem{definition}{Definition}
\newtheorem{theorem}{Theorem}

\numberwithin{figure}{section}
\numberwithin{equation}{section}
\numberwithin{lemma}{section}
\numberwithin{theorem}{section}
\numberwithin{definition}{section}

\begin{document}

\title{Supplementary Material to: Realization of Three-port Spring Networks with Inerter for Effective Mechanical Control}


\author{Michael Z. Q. Chen$^{\dag,*}$,~ Kai Wang$^{\dag}$,  ~ Yun Zou$^{\ddag}$,~ and~ Guanrong Chen$^{\S}$
\thanks{This research was partially supported by the Hong Kong Research Grants Council under the GRF Grant CityU1109/12 and the National Science Foundation of China under Grants 61374053 and 61174038.}
\thanks{$^{\dag}$Department of Mechanical Engineering, The University of Hong Kong, Pokfulam
Road, Hong Kong.}
\thanks{$^{\ddag}$School of Automation, Nanjing University of Science and Technology, Nanjing, P. R. China.}
\thanks{$^{\S}$Department of Electronic Engineering, City University of Hong Kong, Tat Chee Avenue, Kowloon, Hong Kong.}
\thanks{Correspondence: MZQ Chen, mzqchen@hku.hk.}
}

\maketitle

\begin{abstract}
This is a supplementary material to ``Realization of three-port spring networks with inerter for effective mechanical control'' \cite{WCZC15}, which provides the detailed proofs of some results. For more background information, refer to  \cite{Che07}--\cite{Smi02} and references therein.

\medskip

\noindent{\em Keywords:} Passive network synthesis, mechanical network, inerter,  positive-real function, three-port resistive network.

\end{abstract}

\section{INTRODUCTION}  \label{sec: introduction}
This report presents the proofs of some results as well as several other supplementary discussions of the note ``Realization of three-port spring networks with inerter for effective mechanical control'' \cite{WCZC15}.
One assumes that
the numbering of lemmas, theorems, corollaries, and figures
in this report agrees with that in the original note.

\section{Proof of Theorem~1}

\subsection{Preliminary Lemmas}

\begin{definition} \cite{Boe66}
{Considering a mechanical (electrical) network $N_n$ without levers (transformers), a graph $\mathcal{G}$ named \textit{augmented graph} is formulated by letting each port or each element correspond to an \textit{edge} \cite[pg.~9]{SR61} and letting each node of the network correspond to a \textit{vertex} \cite[pg.~9]{SR61}. The subgraph $\mathcal{G}_p$
that consists of all the edges corresponding to the ports is defined as \textit{port graph}. The subgraph $\mathcal{G}_e$  that consists of all the edges corresponding to the elements is defined as \textit{network graph}. }
\end{definition}

It is assumed that $\mathcal{G}$ is \textit{connected} \cite[pg.~15]{SR61}, which does not affect the results by the knowledge of circuit theory and which guarantees the existence of a \textit{tree} \cite[pg.~24]{SR61}. The basic knowledge of graph theory can be referred to \cite{SR61}.

\begin{lemma}    \label{lemma: well-defined}
{An $n$-port resistive  network with a connected augmented graph $\mathcal{G}$ has a well-defined admittance if and only if its port graph $\mathcal{G}_p$ is made part of a tree of $\mathcal{G}$.}
\end{lemma}
\begin{proof}
\textit{Necessity.} Suppose that the admittance of the $n$-port resistive network exists. By \cite[pg.~192]{Boe66}, the voltages of the $n$ ports are independent. Therefore, there is no \textit{circuit} \cite[pg.~15]{SR61} for port graph $\mathcal{G}_p$; otherwise, it would contradict with Kirchhoff's Voltage Law. It immediately follows  from \cite[pg.~27]{SR61} that $\mathcal{G}_p$ must be  part of a tree of the augmented graph $\mathcal{G}$.

\textit{Sufficiency.} $\hat{B}=[B,F]$ denotes the \textit{fundamental circuit matrix} \cite[pg.~91]{SR61} of the augmented graph $\mathcal{G}$, where the columns of $B$ correspond to the elements, and the columns of $F$ correspond to the ports.
It is obvious that the dimension of $B$ is $(n_e + n - n_v + 1)\times n_e$ and the dimension of $F$ is $(n_e + n - n_v + 1)\times n$, where $n_e$ and $n_v$  denote respectively  the number of edges and vertices  of $\mathcal{G}$.  $R$ denotes an $n_e\times n_e$ diagonal matrix, whose diagonal entries are impedances of the elements corresponding to the columns of $B$. Since the port graph $\mathcal{G}_p$ is  part of a tree $\mathcal{T}$ of the augmented graph $\mathcal{G}$, all the \textit{chords}  \cite[pg.~26]{SR61} of $\mathcal{T}$ must be the edges of the network graph $\mathcal{G}_e$. Then, it follows from \cite[pg.~93]{SR61} that $B$ must contain a non-singular submatrix of order $(n_e + n - n_v + 1)$, implying that the rank of $B$ is $(n_e + n - n_v + 1)$. Therefore, $BRB^{T}$ is non-singular. Consequently, based on  the discussion in \cite{Boe66}, the admittance must be well-defined (exist) and must be equal to $F^{T}(BRB^{T})^{-1}F$.
\end{proof}

\begin{definition}  \cite{BT61}
{A tree with all branches incident to a common vertex is named an \textit{L-tree} (\textit{Lagrangian-tree}); a tree whose branches constitute a path is named a \textit{P-tree} (\textit{Path-tree}).}
\end{definition}

\begin{lemma} \cite[pg.~35]{BT61}  \label{lemma: Preliminary L-tree condition}
{A real symmetric $n\times n$ matrix $\mathcal{A}=[a_{ij}]_{n\times n}$ is realizable as the admittance of an $n$-port resistive network, whose augmented graph contains $n+1$ vertices and whose port graph is an L-tree, if and only if the following conditions hold simultaneously: 1) The sign pattern of $\mathcal{A}$ must be such that after a finite number of \textit{cross-sign changes} \cite[pg.~33]{BT61} all the off-diagonal entries of $\mathcal{A}$ are non-positive;
2) $2a_{ii} \geq \sum_{k=1}^n |a_{ik}|$ for all $i \in 1, 2, ... , n$. }
\end{lemma}

From the discussion in  \cite[pp.~34--35]{BT61}, the following lemma can be easily obtained.\footnote{Lemma~\ref{lemma: Preliminary L-tree condition values of elements} is also given in Page~303 of ``E. A. Guillemin, ``On the analysis and synthesis of single-element-kind networks,'' \textit{IRE Trans. Circuit Theory}, vol.~7, no.~3, pp.~303--312, 1960.''}

\begin{lemma}     \label{lemma: Preliminary L-tree condition values of elements}
{If $\mathcal{A}=[a_{ij}]_{n\times n}$ satisfies the conditions of Lemma~\ref{lemma: Preliminary L-tree condition}, then the values of conductances between each pair of vertices (totally $n(n+1)/2$ pairs) are
$|a_{ij}|$ for $i,j \in 1, 2, ... , n$ with $i < j$, and $(2a_{ii} - \sum_{k=1}^n |a_{ik}|)$ for all $i \in 1, 2, ... , n$, which must be uniquely determined.}
\end{lemma}

\begin{lemma} \cite[pg.~34]{BT61}
\label{lemma: Preliminary P-tree condition}
{A real symmetric $n\times n$ matrix $\mathcal{A}=[a_{ij}]_{n\times n}$ is realizable as the admittance of an $n$-port resistive network, whose augmented graph contains $n+1$ vertices and whose port graph is an P-tree, if and only if after a finite number of cross-sign changes and a proper rearrangement of rows and corresponding columns the entries satisfy $a_{i, j-1} - a_{i, j} \geq a_{i-1, j-1} - a_{i-1, j}$, where $a_{0, i} = a_{j, n+1} := 0$, for $i, j \in 1, 2, \ldots, n+1$ with $i < j$.}
\end{lemma}

\begin{lemma}  \cite[pg.~34]{BT61}
\label{lemma: Preliminary P-tree condition values of elements}
{If $\mathcal{A}=[a_{ij}]_{n\times n}$ satisfies $a_{i, j-1} - a_{i, j} \geq a_{i-1, j-1} - a_{i-1, j}$, where $a_{0, i} = a_{j, n+1} := 0$, for all $i, j \in 1, 2, \ldots, n+1$ with $i < j$, then the port edges are ordered and oriented to the same direction, and the values of conductances between each pair of vertices (totally $n(n+1)/2$ pairs) are $g_{i,j} = (a_{i, j-1} - a_{i, j}) - (a_{i-1, j-1} - a_{i-1, j})$, which must be uniquely determined.}
\end{lemma}

\begin{lemma}  \label{lemma: L-tree condition}
{A third-order real  symmetric  matrix $Y_N$ in the form of
\begin{equation} \label{eq: Y_N}
Y_N=\left[
    \begin{array}{ccc}
      y_{11} & y_{12} & y_{13} \\
      y_{12} & y_{22} & y_{23} \\
      y_{13} & y_{23} & y_{33} \\
    \end{array}
  \right]
\end{equation}
can be realized as the admittance of a three-port resistive network with at most three elements, whose augmented graph $\mathcal{G}$ contains four vertices  and  whose port graph $\mathcal{G}_p$ is an L-tree of $\mathcal{G}$, if and only if the following conditions hold simultaneously: 1)~$y_{12}y_{13}y_{23}\leq 0$; 2)~$y_{11}-|y_{12}|-|y_{13}|\geq 0$, $y_{22}-|y_{12}|-|y_{23}| \geq 0$, and $y_{33}-|y_{13}|-|y_{23}|\geq 0$; 3)~at least three of
$y_{12}$, $y_{13}$, $y_{23}$, $y_{11}-|y_{12}| - |y_{13}|$, $y_{22}-|y_{12}|-|y_{23}|$,
and $y_{33}-|y_{13}|-|y_{23}|$ are zero.
}
\end{lemma}
\begin{proof}
\textit{Necessity.}
It is obvious that if a third-order real  symmetric  matrix $Y_N$ is realizable by the three-port resistive network with at most three elements, whose augmented graph $\mathcal{G}$ contains four vertices  and  whose port graph $\mathcal{G}_p$ is an L-tree of $\mathcal{G}$, then $Y_N$ satisfies the conditions of Lemma~\ref{lemma: Preliminary L-tree condition} for $n = 3$, and at least three of conductances between each of vertices are zero.
From the latter constraint, Condition~3 of this lemma is directly obtained based on Lemma~\ref{lemma: Preliminary L-tree condition values of elements}. Moreover, it is easy to check that a finite number of cross-sign changes does not change the fact that $a_{12}a_{13}a_{23} \leq 0$ when $n = 3$. Therefore, Conditions~1 and 2 of this lemma are obtained.

\textit{Sufficiency.}
Since $Y_N$ satisfies Condition~1, there must exist a finite number of cross-sign changes such that $y_{12} \leq 0$, $y_{13} \leq 0$, and $y_{23} \leq 0$. Together with Condition~2, it follows that $Y_N$ satisfies the conditions of Lemma~\ref{lemma: Preliminary L-tree condition}. Due to Lemma~\ref{lemma: Preliminary L-tree condition values of elements}, the values of elements are $|y_{12}|$, $|y_{13}|$, $|y_{23}|$, $(y_{11}-|y_{12}| - |y_{13}|)$, $(y_{22}-|y_{21}|-|y_{23}|)$, and $(y_{33}-|y_{31}|-|y_{32}|)$. From Condition~3,  $Y_N$ is realizable by the three-port resistive network with at most three elements, whose augmented graph $\mathcal{G}$ contains four vertices  and  whose port graph $\mathcal{G}_p$ is an L-tree of $\mathcal{G}$.
\end{proof}

\begin{lemma}  \label{lemma: P-tree condition}
{A third-order real symmetric matrix $Y_N$ in the form of \eqref{eq: Y_N} can be realized as the admittance of a
three-port resistive  network with at most three  elements, whose augmented graph $\mathcal{G}$ contains four
vertices and whose port graph $\mathcal{G}_p$ is a P-tree of $\mathcal{G}$, if and only if $y_{12}y_{13}y_{23} \geq 0$ and at least one of the following three conditions holds with at least three of the six inequality signs being equality:
1)~$-|y_{13}| \leq 0$, $|y_{13}| \leq |y_{12}| \leq y_{11}$, $|y_{13}| \leq |y_{23}| \leq y_{33}$, and
$|y_{12}| + |y_{23}| - |y_{13}| \leq y_{22}$;
2)~$-|y_{12}| \leq 0$, $|y_{12}| \leq |y_{13}| \leq y_{11}$, $|y_{12}| \leq |y_{23}| \leq y_{22}$, and
$|y_{13}| + |y_{23}| - |y_{12}| \leq y_{33}$;
3)~$-|y_{23}| \leq 0$, $|y_{23}|\leq |y_{12}| \leq y_{22}$,  $|y_{23}| \leq |y_{13}| \leq y_{33}$, and
$|y_{12}| + |y_{13}| - |y_{23}| \leq y_{11}$.
}
\end{lemma}
\begin{proof}
\textit{Necessity.}
It is obvious that if a third-order real  symmetric  matrix $Y_N$ is realizable by the three-port resistive network with at most three elements, whose augmented graph $\mathcal{G}$ contains four vertices  and  whose port graph $\mathcal{G}_p$ is an P-tree of $\mathcal{G}$, then $Y_N$ satisfies the conditions of Lemma~\ref{lemma: Preliminary P-tree condition} for $n = 3$, and at least three of conductances between each of vertices are zero. Together with Lemma~\ref{lemma: Preliminary P-tree condition values of elements}, the conditions of this lemma are obtained.

\textit{Sufficiency.} There must exist a finite number of cross-sign changes and a proper rearrangement of rows and corresponding columns such that $y_{11} - y_{12} \geq 0$, $y_{12} - y_{13} \geq 0$, $y_{13} \geq 0$, $(y_{22} - y_{23}) - (y_{12} - y_{13}) \geq 0$, $y_{23} - y_{13} \geq 0$, and $y_{33} - y_{23} \geq 0$ with at least three of the six inequality signs being equality. Therefore, the condition of Lemma~\ref{lemma: Preliminary P-tree condition} must hold for $n = 3$. Together with Lemma~\ref{lemma: Preliminary P-tree condition values of elements}, $Y_N$ is realizable by the three-port resistive network with at most three elements, whose augmented graph $\mathcal{G}$ contains four vertices  and  whose port graph $\mathcal{G}_p$ is an P-tree of $\mathcal{G}$.
\end{proof}

\begin{lemma}  \label{lemma: diagonal paramount}
{If a third-order real paramount  matrix $Y_N$ in the form of \eqref{eq: Y_N} is a diagonal matrix, then the conditions of Lemma~\ref{lemma: L-tree condition} must hold.}
\end{lemma}
\begin{proof}
It is obvious.
\end{proof}

\begin{lemma}  \label{lemma: diagnoal entries being zero}
{For a third-order real paramount matrix $Y_N$ in the form of \eqref{eq: Y_N}, if any one of the diagonal entries $y_{11}$, $y_{22}$, and $y_{33}$ is zero, then the conditions of Lemma~\ref{lemma: L-tree condition} must hold.}
\end{lemma}
\begin{proof}
Assuming that $y_{33}=0$, the paramountcy of $Y_N$ implies that $y_{13}=y_{23}=0$ and $y_{11} \geq |y_{12}| \geq 0$. Hence, one obtains that $y_{12}y_{13}y_{23}= 0$, $y_{13}=0$, $y_{23}=0$, $y_{11}-|y_{12}| - |y_{13}| = y_{11} - |y_{12}| \geq 0$, $y_{22}-|y_{12}|-|y_{23}| = y_{22} - |y_{12}| \geq 0$, and $y_{33}-|y_{13}|-|y_{23}| = 0$. Thus, the conditions of Lemma~\ref{lemma: L-tree condition} hold. Similarly, the case of $y_{11}=0$ or $y_{22}=0$ can be proved.
\end{proof}

\begin{lemma}   \label{lemma: two rows equal or opposite}
{For a third-order real paramount matrix $Y_N$ in the form of \eqref{eq: Y_N}, if there exist two equal rows or two rows for which one row is the negative of the other, then the conditions of Lemma~\ref{lemma: P-tree condition} must hold.}
\end{lemma}
\begin{proof}
First, consider the case with two equal rows. If they are assumed to be the first and second rows, together with the symmetry, $Y_N$ can be written as
\begin{equation*}
Y_N = \left[
        \begin{array}{ccc}
          y_{11} & y_{11} & y_{13} \\
          y_{11} & y_{11} & y_{13} \\
          y_{13} & y_{13} & y_{33} \\
        \end{array}
      \right].
\end{equation*}
Since $Y_N$ is paramount, it follows that $y_{12}y_{13}y_{23}=y_{11}y_{13}^2 \geq 0$,
$-|y_{13}| \leq 0$, $y_{11}-|y_{12}|=0$, $|y_{12}|-|y_{13}|=y_{11}-|y_{13}|\geq 0$,
$y_{33}-|y_{23}|=y_{33}-|y_{13}| \geq 0$, $|y_{23}|-|y_{13}| = 0$, and $y_{22} + |y_{13}| - (|y_{12}| + |y_{23}|) = y_{11} + |y_{13}| - (y_{11} + |y_{13}|) = 0$. Hence, the conditions of Lemma~\ref{lemma: P-tree condition} must hold.

When there is another pair of two equal rows,
properly arranging the rows and the corresponding columns of $Y_N$  can always yield another paramount matrix $Y'_N$, whose first and second rows are equal. Hence, $Y'_N$ satisfies the conditions of Lemma~\ref{lemma: P-tree condition} as discussed above. Since arranging the rows and the corresponding columns means swapping the ports of the realization network, one implies that $Y_N$ also satisfies the conditions of Lemma~\ref{lemma: P-tree condition}.

Now, let us to prove the other case with two rows for which one row is the negative of the other.
If they are assumed to be the first and second rows, together with the symmetry, $Y_N$ can be written as
\begin{equation*}
Y_N = \left[
        \begin{array}{ccc}
          y_{11} & -y_{11} & y_{13} \\
          -y_{11} & y_{11} & -y_{13} \\
          y_{13} & -y_{13} & y_{33} \\
        \end{array}
      \right].
\end{equation*}
Since $Y_N$ is paramount, it follows that $y_{12}y_{13}y_{23}=y_{11}y_{13}^2 \geq 0$, $-|y_{13}| \leq 0$,
$y_{11}-|y_{12}| = 0$, $|y_{12}|-|y_{13}|=y_{11}-|y_{13}|\geq 0$, $y_{33}-|y_{23}|=y_{33}-|y_{13}| \geq 0$, $|y_{23}|-|y_{13}| = 0$, and $y_{22} + |y_{13}| - (|y_{12}| + |y_{23}|) = y_{11} + |y_{13}| - (y_{11} + |y_{13}|) = 0$.  Hence, the conditions of Lemma~\ref{lemma: P-tree condition} must hold. When it comes to other two rows, arranging them properly such that they become the first two rows can prove this.
\end{proof}

\subsection{Main Proof}

\textit{Theorem~1:}
A third-order real symmetric matrix $Y_N$ in the form of \eqref{eq: Y_N} can be realized as the admittance of a three-port resistive network with at most three  elements if and only if it satisfies the conditions of Lemma~\ref{lemma: L-tree condition} or Lemma~\ref{lemma: P-tree condition}, that is, if and only if one of the following two conditions holds:
\begin{enumerate}
  \item[1.] $y_{12}y_{13}y_{23}\leq 0$, $y_{11}-|y_{12}|-|y_{13}|\geq 0$, $y_{22}-|y_{12}|-|y_{23}| \geq 0$, $y_{33}-|y_{13}|-|y_{23}|\geq 0$, and at least three of $y_{12}$, $y_{13}$, $y_{23}$, $y_{11}-|y_{12}| - |y_{13}|$, $y_{22}-|y_{12}|-|y_{23}|$, and $y_{33}-|y_{13}|-|y_{23}|$ are zero.
  \item[2.] $y_{12}y_{13}y_{23} \geq 0$ and at least one of the following three conditions holds with at least three of the six inequality signs being equality:
      a)~$-|y_{13}| \leq 0$, $|y_{13}| \leq |y_{12}| \leq y_{11}$,
      $|y_{13}| \leq |y_{23}| \leq y_{33}$, and
      $|y_{12}| + |y_{23}| - |y_{13}| \leq y_{22}$;
      b)~$-|y_{12}| \leq 0$, $|y_{12}| \leq |y_{13}| \leq y_{11}$,
      $|y_{12}| \leq |y_{23}| \leq y_{22}$, and
      $|y_{13}| + |y_{23}| - |y_{12}| \leq y_{33}$;
      c)~$-|y_{23}| \leq 0$, $|y_{23}|\leq |y_{12}| \leq y_{22}$,  $|y_{23}| \leq |y_{13}| \leq y_{33}$, and
      $|y_{12}| + |y_{13}| - |y_{23}| \leq y_{11}$.
\end{enumerate}

\begin{proof}
\textit{Sufficiency.}
Lemmas~\ref{lemma: L-tree condition} and \ref{lemma: P-tree condition} together imply the sufficiency.

\textit{Necessity.}
Since the augmented graph $\mathcal{G}$ contains at most six edges and the port graph $\mathcal{G}_p$ must be made part of a tree,
the number of vertices ranges from four to seven. If $\mathcal{G}$ contains four vertices, then $\mathcal{G}_p$ is either an L-tree or a P-tree, which implies the conditions of Lemma~\ref{lemma: L-tree condition} or Lemma~\ref{lemma: P-tree condition}.
If $\mathcal{G}$ contains seven vertices, then all the edges of $\mathcal{G}$ constitute a tree.
Hence, currents of the ports must be zero, which implies that the conditions of Lemma~\ref{lemma: L-tree condition} hold  due to Lemma~\ref{lemma: diagnoal entries being zero}. Hence, it remains to consider networks whose augmented graphs contain five or six vertices, and to show none of their admittances can simultaneously contradict the conditions of Lemmas~\ref{lemma: L-tree condition} and \ref{lemma: P-tree condition}.

Suppose that $\mathcal{G}$ contains five vertices. Then, any tree of $\mathcal{G}$ must contain four edges (totally three possibilities as in Fig.~\ref{fig: port-tree}), and the number of chords is at most two. Suppose that $\mathcal{T}$ is a tree of $\mathcal{G}$ containing $\mathcal{G}_p$.
If neither the conditions of Lemma~\ref{lemma: L-tree condition} nor those of Lemma~\ref{lemma: P-tree condition} hold, then $\mathcal{G}$, $\mathcal{G}_p$, and $\mathcal{T}$ must satisfy the following four properties:
\begin{enumerate}
\item[1.] If $\mathcal{G}$ is \textit{separable} \cite[pg.~35]{SR61}, then each \textit{component} \cite[pg.~38]{SR61} contains at least one edge belonging to $\mathcal{G}_p$;
\item[2.] $\mathcal{G}$ is either nonseparable or only contains two \textit{cyclically connected} \cite[pg.~37]{SR61} components;
\item[3.] all end vertices (such as $a$, $d$, $e$ of Fig.~\ref{fig: port-tree}(b)) of $\mathcal{T}$ must be incident with at least one chord;
\item[4.] $\mathcal{G}$ has at most one pair of series edges both belonging to $\mathcal{T}$ and $\mathcal{G}_p$.
\end{enumerate}
The reason is as follows. If Property~1 does not hold, then it is obvious that the network can be equivalent to another one containing fewer nodes and elements because of zero currents, whose admittance must satisfy the conditions of Lemma~\ref{lemma: L-tree condition} or Lemma~\ref{lemma: P-tree condition}. Property~1 further implies that the number of components cannot be more than three. If there are three components, then the admittance of the network must be diagonal, which must satisfy the conditions of Lemma~\ref{lemma: L-tree condition} by Lemma~\ref{lemma: diagonal paramount}. If Property~1 does not hold and there exists a component that is not cyclically connected, then $\mathcal{G}$ contains at least one end branch belonging to $\mathcal{G}_p$, which implies the conditions of Lemma~\ref{lemma: L-tree condition} due to Lemma~\ref{lemma: diagnoal entries being zero}. This proves Property~2. Similarly, Property~3 can be verified. By Lemma~\ref{lemma: two rows equal or opposite}, the admittance of the required networks with two ports in series must satisfy the conditions of Lemma~\ref{lemma: P-tree condition}, which implies Property~4.

\begin{figure}[thpb]
      \centering
      \subfigure[]{
      \includegraphics[scale=0.55]{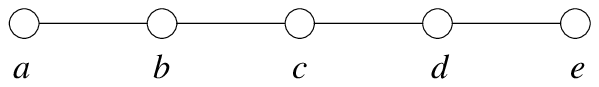}
      \label{subfig: port-tree-01}}
      \hspace{-0.2cm}
      \subfigure[]{
      \includegraphics[scale=0.55]{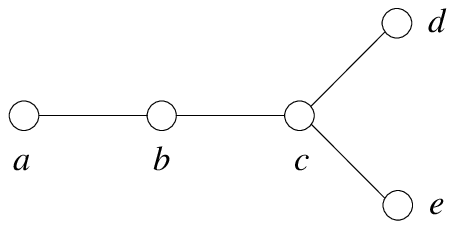}
      \label{subfig: port-tree-02}}
      \hspace{-0.2cm}
      \subfigure[]{
      \includegraphics[scale=0.55]{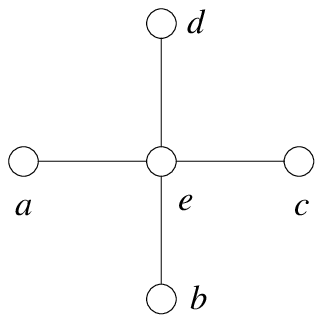}
      \label{subfig: port-tree-03}}
      \caption{The trees $\mathcal{T}$ containing port graph $\mathcal{G}_p$ when the number of vertices is five.}
      \label{fig: port-tree}
   \end{figure}

Then, the above four properties are utilized to eliminate cases when the conditions of Lemma~\ref{lemma: L-tree condition} or Lemma~\ref{lemma: P-tree condition} hold, which simplifies the discussion. By Properties~3--4, one can verify that the number of chords should not be less than two. It is not difficult to find  the augmented graphs $\mathcal{G}$ that satisfy Properties~1--4. It is known that if the corresponding networks can be equivalent to the one with four nodes and at most three elements, then the conditions of Lemma~\ref{lemma: L-tree condition} or Lemma~\ref{lemma: P-tree condition} must hold. As a consequence, by making use of the equivalence of two series (or parallel) elements or the generalized star-mesh transformation, one can show that none of the networks corresponding to these graphs can simultaneously contradict the conditions of Lemma~\ref{lemma: L-tree condition} and Lemma~\ref{lemma: P-tree condition}.

Suppose that $\mathcal{G}$ contains six vertices. Then, any tree of $\mathcal{G}$ must contain five edges, and the number of chords is at most one. For such a network that cannot be equivalent to the one whose augmented graph contains fewer vertices, $\mathcal{G}$, $\mathcal{G}_p$, and $\mathcal{T}$ must also satisfy Properties~1--4 as discussed above. Then, Property~3 implies that the only possible port graph must be the P-tree since there is at most one chord. Together with Property~4, one implies that no such a network exists.

The necessity  is thus proved.
\end{proof}

\section{Proof of Theorem~5}

\subsection{Previous Lemmas}

\begin{lemma}  \label{lemma: coefficients 01}
{A matrix $G$ as defined in
\begin{equation}  \label{eq: G}
G:= \left[
      \begin{array}{ccc}
        G_1 & G_4 & G_5 \\
        G_4 & G_2 & G_6 \\
        G_5 & G_6 & G_3 \\
      \end{array}
    \right]
\end{equation}
is non-negative definite if and only if
$\alpha_0$, $\alpha_1$, $\alpha_2$, $\alpha_3$, $\beta_1$, $\beta_2$, $\beta_3$ $\geq 0$ as defined in
\begin{equation}  \label{eq: matrix condition}
\begin{split}
&\alpha_3 = G_1, \ \alpha_2 = G_1G_2 - G_4^2, \ \alpha_1 = G_1G_3 - G_5^2,    \\
\alpha_0 &= \det(G), \ \beta_3 = G_2, \ \beta_2 = G_3, \ \beta_1 = G_2G_3 - G_6^2.
\end{split}
\end{equation} }
\end{lemma}
\begin{proof}
The matrix $G$ is non-negative if and only if
\begin{equation*}
\begin{split}
G_1=\alpha_3 \geq 0, \ &G_2=\beta_3 \geq 0, \  G_3=\beta_2 \geq 0, \\
\left|\begin{array}{cc} G_1 & G_4 \\
G_4 & G_2 \end{array}\right|=&\alpha_2\geq 0,
\left|\begin{array}{cc} G_2 & G_6 \\
G_6 & G_3 \end{array}\right| =\beta_1 \geq 0,  \\
\left|\begin{array}{cc} G_1 & G_5 \\
G_5 & G_3 \end{array}\right|=\alpha_1 &\geq 0, \ \left|\begin{array}{ccc} G_1 & G_4 & G_5 \\
G_4 & G_2 & G_6 \\
G_5 & G_6 & G_3\end{array}\right|=\alpha_0 \geq 0.
\end{split}
\end{equation*}
Hence, this lemma is proved.
\end{proof}

\begin{lemma}  \label{lemma: one funtion be expressed by the other}
{Consider any function $Y(s)$ in the form of
\begin{equation}   \label{eq: Y initial}
Y(s)  = \frac{\alpha_3 s^3 + \alpha_2 s^2 + \alpha_1 s + \alpha_0}{ s^4 + \beta_3 s^3 + \beta_2 s^2 + \beta_1 s },
\end{equation}
where $\alpha_0$, $\alpha_1$, $\alpha_2$, $\alpha_3$, $\beta_1$, $\beta_2$, $\beta_3$ $\geq 0$.
$Y(s)$ can also be expressed as
\begin{equation}  \label{eq: Y general}
Y(s) = \frac{G_1 s^3 + (G_1 G_2 - G_4^2)s^2 + (G_1G_3 - G_5^2)s + \det(G)}{s\left( s^3 + G_2s^2 + G_3s + (G_2G_3 - G_6^2) \right)}
\end{equation}
with non-negative definite $G$ defined in \eqref{eq: G} and the entries of $G$ satisfying \eqref{eq: matrix condition}
if and only if $W_1$, $W_2$, $W_3$ $\geq 0$ and $W^2 = 4W_1W_2W_3$.
}
\end{lemma}
\begin{proof}
\textit{Sufficiency.} It suffices to show that one can always find a non-negative $G$ as defined in \eqref{eq: G} such that admittance \eqref{eq: Y general} concerned with the entries of $G$ equals the given admittance with non-negative coefficients.
Let $G_1 = \alpha_3$, $G_2 = \beta_3$, and $G_3 = \beta_2$. Furthermore, $W_1$, $W_2$, $W_3$ $\geq 0$ guarantees that one can always find $G_4^2$, $G_5^2$, and $G_6^2$ such that
\begin{align*}
G_4^2 &= W_1 = \alpha_3\beta_3 - \alpha_2 \ \Rightarrow \ \alpha_2 = G_1G_2 - G_4^2,  \\
G_5^2 &= W_2 = \alpha_3\beta_2 - \alpha_1 \ \Rightarrow \
\alpha_1 = G_1G_3 - G_5^2,  \\
G_6^2 &= W_3 = \beta_2\beta_3 - \beta_1 \ \Rightarrow \
\beta_1 = G_2G_3 - G_6^2.
\end{align*}
It is seen that the signs of $G_4$, $G_5$, and $G_6$ have not yet been fixed, and
$
G_4G_5G_6 = \pm \sqrt{W_1W_2W_3}.
$
Since $W^2 = 4W_1W_2W_3$ holds, one has $W = \pm 2 \sqrt{W_1W_2W_3}$. If $W = 2 \sqrt{W_1W_2W_3}$, then choose the signs of $G_4$, $G_5$, and $G_6$ to make $G_4G_5G_6 = \sqrt{W_1W_2W_3}$; otherwise let $G_4G_5G_6 = - \sqrt{W_1W_2W_3}$. Hence, one obtains $2G_4G_5G_6 = W$, which implies
\begin{equation*}
\begin{split}
\alpha_0 =& 2 G_4G_5G_6 - 2\alpha_3\beta_2\beta_3 + \alpha_3\beta_1 + \alpha_2\beta_2 + \alpha_1\beta_3   \\
=& 2 G_4G_5G_6 + \alpha_3\beta_2\beta_3 - \alpha_3(\beta_2\beta_3 - \beta_1) \\
&- \beta_3(\alpha_3\beta_2 - \alpha_1) - \beta_2(\alpha_3\beta_3 - \alpha_2)   \\
=& 2G_4G_5G_6 + G_1G_2G_3 - G_1G_6^2 - G_2G_5^2 - G_3G_4^2 = \det(G).
\end{split}
\end{equation*}
Now, one can see that \eqref{eq: matrix condition} holds. Therefore, admittance \eqref{eq: Y general} concerned with $G_1$, $G_2$, $G_3$, $G_4$, $G_5$, and $G_6$ is equal to the given admittance with $\alpha_0$, $\alpha_1$, $\alpha_2$, $\alpha_3$, $\beta_1$, $\beta_2$, $\beta_3$ $\geq 0$. Besides, $G$ as in \eqref{eq: G}  is non-negative definite based on Lemma~\ref{lemma: coefficients 01}.

\textit{Necessity.} Since one can always find a non-negative definite matrix $G$ as defined in \eqref{eq: G} to make \eqref{eq: matrix condition} hold, it is calculated that $W_1 = \alpha_3\beta_3 - \alpha_2 = G_4^2 \geq 0$, $W_2 = \alpha_3\beta_2 - \alpha_1 = G_5^2 \geq 0$, $W_3 = \beta_2 \beta_3 - \beta_1 = G_6^2 \geq 0$, and $W^2 = 4W_1W_2W_3 = 4G_4^2 G_5^2 G_6^2$.
\end{proof}

\begin{lemma}   \label{lemma: condition 01}
{Consider a non-negative definite matrix $G$ in the form of \eqref{eq: G} and the variables $\alpha_0$, $\alpha_1$, $\alpha_2$, $\alpha_3$, $\beta_1$, $\beta_2$, $\beta_3$ as defined in \eqref{eq: matrix condition}. There exists at least one of  the first-order minors or second-order minors of $G$ being zero if and only if at least one of $\alpha_1$, $\alpha_2$, $\alpha_3$, $\beta_1$, $\beta_2$, $\beta_3$, $W_1$, $W_2$, $W_3$, $\left(\beta_2-W/(2W_1)\right)$, $\left(\beta_3-W/(2W_2)\right)$, and
$\left(\alpha_3-W/(2W_3)\right)$ is zero.}
\end{lemma}
\begin{proof}
Since \eqref{eq: matrix condition} holds, one obtains $G_1 = \alpha_3$, $G_2 = \beta_3$, $G_3 = \beta_2$, $G_4^2 = W_1$, $G_5^2 = W_2$, $G_6^2 = W_3$, and $2G_4G_5G_6 = W$. It then follows that
\begin{align*}
G_1 - \frac{G_4G_5}{G_6} &= G_1 - \frac{2G_4G_5G_6}{2G_6^2} = \alpha_3 - \frac{W}{2W_3},   \\
G_2 - \frac{G_4G_6}{G_5} &= G_2 - \frac{2G_4G_5G_6}{2G_5^2} =
\beta_3 - \frac{W}{2W_2},    \\
G_3 - \frac{G_5G_6}{G_4} &= G_3 - \frac{2G_4G_5G_6}{2G_4^2} =
\beta_2 - \frac{W}{2W_1}.
\end{align*}
It can be seen that there exists at least one of the first-order minors or second-order minors being zero if and only if at least one of the following twelve equations holds:
\begin{equation*}
\begin{split}
G&_1 = 0, \ G_2 = 0, \ G_3 = 0, \ G_4 = 0, \ G_5 = 0, \ G_6 = 0, \\
&G_1G_2 - G_4^2 = 0, \ G_1G_3 - G_5^2 = 0, \ G_2G_3 - G_6^2 = 0, \\
G_1G_6& - G_4G_5 = 0, \ G_4G_6 - G_2G_5 = 0, \ G_3G_4 - G_5G_6 = 0.
\end{split}
\end{equation*}
For this equivalent condition, one has the following relations:
\begin{equation*}
\begin{split}
G_1 = 0 \Leftrightarrow &\alpha_3 = 0, \ G_2 = 0 \Leftrightarrow \beta_3 = 0, \ G_3 = 0 \Leftrightarrow \beta_2 = 0,  \\
G_4 = 0 \Leftrightarrow W_1 &= 0, \ G_5 = 0 \Leftrightarrow W_2 = 0, \ G_6 = 0 \Leftrightarrow W_3 = 0,  \\
G_1G_2 - G_4^2 &= 0 \Leftrightarrow \alpha_2 = 0, \ G_1G_3 - G_5^2 = 0 \Leftrightarrow \alpha_1 = 0,  \\
&G_2G_3 - G_6^2 = 0 \Leftrightarrow \beta_1 = 0.
\end{split}
\end{equation*}
When $G_4G_5G_6 \neq 0$, the following relations are also satisfied:
\begin{align*}
G_1G_6 - G_4G_5 = 0 \Leftrightarrow G_1 = \frac{2G_4G_5G_6}{2G_6^2} \Leftrightarrow \alpha_3 = \frac{W}{2W_3},  \\
G_4G_6 - G_2G_5 = 0 \Leftrightarrow G_2 = \frac{2G_4G_5G_6}{2G_5^2} \Leftrightarrow \beta_3 = \frac{W}{2W_2},  \\
G_3G_4 - G_5G_6 = 0 \Leftrightarrow G_3 = \frac{2G_4G_5G_6}{2G_4^2} \Leftrightarrow \beta_2 = \frac{W}{2W_1}.
\end{align*}
Combining together the above discussions, this lemma is proved.
\end{proof}

\begin{lemma}  \label{lemma: condition 02}
{Consider a non-negative definite matrix $G$ as defined in \eqref{eq: G}, whose  all the first-order minors and second-order minors
are all non-zero, with variables $\alpha_1$, $\alpha_2$, $\alpha_3$, $\beta_1$, $\beta_2$, $\beta_3$,
and $\beta_4$ as defined in \eqref{eq: matrix condition}. Then, $G$ satisfies the condition of
Lemma~3 if and only if one of the following
holds:
1)~$W < 0$ and $\alpha_0 = 0$; 2)~$W > 0$ and $\alpha_0 + \alpha_3\beta_1 + \alpha_2\beta_2 - \alpha_1\beta_3 = 0$; 3)~$W > 0$ and $\alpha_0 + \alpha_3\beta_1 + \alpha_1\beta_3 - \alpha_2\beta_2 = 0$;
4)~$W > 0$ and $\alpha_0 + \alpha_1\beta_3 + \alpha_2\beta_2 - \alpha_3\beta_1 = 0$.
}
\end{lemma}
\begin{proof}
Since \eqref{eq: matrix condition} holds, it can be calculated that $2G_4G_5G_6 = W$ and
\begin{equation*}
\begin{split}
G_1G_2G_3 + G_4G_5G_6 &- G_1G_6^2 - G_3G_4^2   \\
& =\frac{\alpha_0 + \alpha_3\beta_1 + \alpha_2\beta_2 - \alpha_1\beta_3}{2},
\end{split}
\end{equation*}
\begin{equation*}
\begin{split}
G_1G_2G_3 + G_4G_5G_6 &- G_1G_6^2 - G_2G_5^2   \\
& = \frac{\alpha_0 + \alpha_3\beta_1 + \alpha_1\beta_3 - \alpha_2\beta_2}{2},
\end{split}
\end{equation*}
\begin{equation*}
\begin{split}
G_1G_2G_3 + G_4G_5G_6 &- G_2G_5^2 - G_3G_4^2    \\
& = \frac{\alpha_0 + \alpha_1\beta_3 + \alpha_2\beta_2 - \alpha_3\beta_1}{2}.
\end{split}
\end{equation*}
Now, the lemma can be proved.
\end{proof}

\subsection{Main Proof}

\textit{Theorem~5:}
A positive-real function $Y(s)$ can be realized as the
driving-point admittance of a one-port network,  consisting of one damper, one inerter, and at most three springs, and satisfying Assumption~1, if and only if $Y(s)$ can be written in the form of \eqref{eq: Y initial},
where the coefficients satisfy $\alpha_0$, $\alpha_1$, $\alpha_2$, $\alpha_3$, $\beta_1$, $\beta_2$, $\beta_3$ $\geq 0$, $W_1$, $W_2$, $W_3$ $\geq 0$, $W^2 = 4W_1W_2W_3$, and also satisfy
1) at least one of $\alpha_1$, $\alpha_2$, $\alpha_3$, $\beta_1$, $\beta_2$, $\beta_3$, $W_1$, $W_2$, $W_3$, $\left(\beta_2-W/(2W_1)\right)$, $\left(\beta_3-W/(2W_2)\right)$, and
$\left(\alpha_3-W/(2W_3)\right)$ is zero; or 2) one of the following
holds with Condition~1 being not satisfied: a)~$W < 0$ and $\alpha_0 = 0$; b)~$W > 0$ and $\alpha_0 + \alpha_3\beta_1 + \alpha_2\beta_2 - \alpha_1\beta_3 = 0$; c)~$W > 0$ and $\alpha_0 + \alpha_3\beta_1 + \alpha_1\beta_3 - \alpha_2\beta_2 = 0$;
d)~$W > 0$ and $\alpha_0 + \alpha_1\beta_3 + \alpha_2\beta_2 - \alpha_3\beta_1 = 0$.
\begin{proof}
\textit{Necessity.} Due to Theorem~2, it is seen that $Y(s)$ is written in the form of \eqref{eq: Y general} with non-negative definite $G$ as defined in \eqref{eq: G}. Therefore, $Y(s)$ can also be expressed as \eqref{eq: Y initial} with coefficients satisfying \eqref{eq: matrix condition}. Due to Lemma~\ref{lemma: coefficients 01}, one has $\alpha_0$, $\alpha_1$, $\alpha_2$, $\alpha_3$, $\beta_1$, $\beta_2$, $\beta_3$ $\geq 0$. After calculations based on \eqref{eq: matrix condition}, it is obtained that $W_1$, $W_2$, $W_3$ $\geq 0$  and $W^2 = 4W_1W_2W_3$. Since $G$ as defined in \eqref{eq: G} satisfies the conditions of Lemma~2 or Lemma~3, it then follows that $\alpha_0$, $\alpha_1$, $\alpha_2$, $\alpha_3$, $\beta_1$, $\beta_2$, $\beta_3$ must satisfy the conditions of
Lemma~\ref{lemma: condition 01} or
Lemma~\ref{lemma: condition 02} (Condition~1 or Condition~2 of this theorem). Hence, the necessity part is proved.

\textit{Sufficiency.} From Lemma~\ref{lemma: one funtion be expressed by the other}, $Y(s)$ can also be expressed as \eqref{eq: Y general} with non-negative definite $G$ defined in \eqref{eq: G}. Furthermore, coefficients $\alpha_0$, $\alpha_1$, $\alpha_2$, $\alpha_3$, $\beta_1$, $\beta_2$, $\beta_3$ satisfy \eqref{eq: matrix condition}. If the conditions of Lemma~\ref{lemma: condition 01} (Condition~1 of this theorem) hold, then there must exist at least one of the first-order minors or second-order minors of $G$ being zero due to that Lemma, which is the conditions of Lemma~2. If the conditions of Lemma~\ref{lemma: condition 02} hold and the conditions of Lemma~\ref{lemma: condition 01} do not hold (Condition~2 of this theorem), then $G$ must satisfy the conditions of Lemma~3. Finally, based on Theorem~2, $Y(s)$ is realizable by the network with one damper, one inerter, and an arbitrary number of springs.
For the case when Lemma~2   holds,
based on Theorem~3, $Y(s)$ is realizable by series-parallel
networks containing at most three springs, one damper, and one inerter, through the Foster Preamble.
For the case when Lemma~3   holds, $Y(s)$ is
realizable by one of four networks shown in Fig.~\ref{fig: covering-configuration}. The expressions of values of elements
can be converted from those in Theorem~4 through relations
\eqref{eq: matrix condition}. Since it is calculated that $G_4^2 = W_1$, $G_5^2 = W_2$, $G_6^2 = W_3$,
$2G_4G_5G_6 = W$, $G_1 - G_4G_5/G_6 = \alpha_3 - W/(2W_3)$, $G_2 - G_4G_6/G_5 = \beta_3 - W/(2W_2)$, and
$G_3 - G_5G_6/G_4 = \beta_2 - W/(2W_1)$, the values can be easily obtained as follows.
\begin{figure}[thpb]
     \centering
     \subfigure[]{
      \includegraphics[scale=0.76]{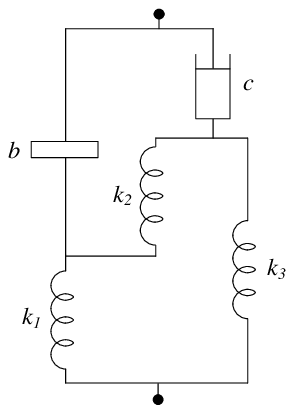}
      \label{fig: first-configuration}}
      \hspace{-0.52cm}
      \subfigure[]{
      \includegraphics[scale=0.76]{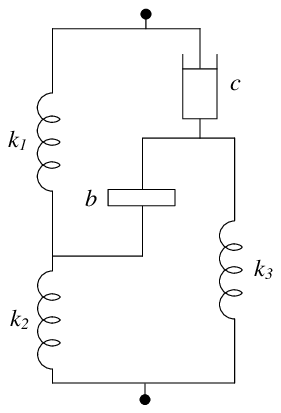}
      \label{fig: second-configuration}}
      \hspace{-0.52cm}
      \subfigure[]{
      \includegraphics[scale=0.76]{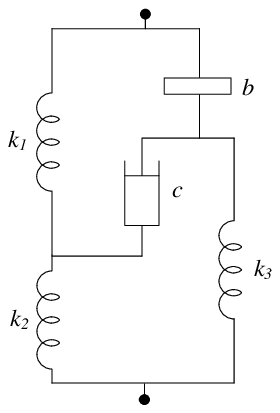}
      \label{fig: third-configuration}}
      \hspace{-0.52cm}
      \subfigure[]{
      \includegraphics[scale=0.76]{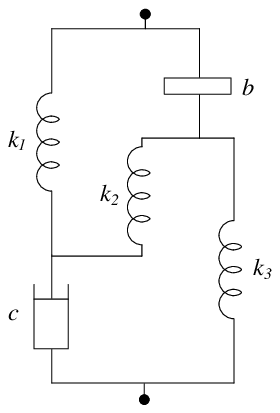}
      \label{fig: fourth-configuration}}
      \caption{The configurations covering all cases that satisfy the conditions  of Lemma~3 (Condition~2 of Theorem~5).}
      \label{fig: covering-configuration}
\end{figure}
\begin{enumerate}
  \item[a)] If $W < 0$ and $\alpha_0 = 0$, then $Y(s)$ can be realized in the form of  Fig.~\ref{fig: covering-configuration}(a), with
      \begin{align}
      k_1 &= \frac{\alpha_2}{\beta_3 - \frac{W}{2W_2}}, \label{eq: first-configuration-k1-02}  \\
      k_2 &= \frac{W_3(\alpha_3 - \frac{W}{2W_3})\alpha_2}{W_2(\beta_3 - \frac{W}{2W_2})^2},  \label{eq: first-configuration-k2-02} \\
      k_3 &= \frac{W(\frac{W}{2W_3} - \alpha_3)}{2W_2(\beta_3 - \frac{W}{2W_2})},  \label{eq: first-configuration-k3-02} \\
      b &= \frac{\alpha_2^2}{W_2(\beta_3 - \frac{W}{2W_2})^2},  \label{eq: first-configuration-b-02} \\
      c &= \frac{W_3 (\alpha_3 - \frac{W}{2W_3})^2}{W_2(\beta_3 - \frac{W}{2W_2})^2}. \label{eq: first-configuration-c-02}
      \end{align}
  \item[b)] If $W > 0$ and $\alpha_0 + \alpha_3\beta_1 + \alpha_2\beta_2 - \alpha_1\beta_3 = 0$, then $Y(s)$ can be realized in the form of
      Fig.~\ref{fig: covering-configuration}(b), with
      \begin{align}
      k_1 &= \frac{\alpha_3 W_3 (\alpha_3 - \frac{W}{2W_3})}{\beta_2 W_1}, \label{eq: third-configuration-k1-02}  \\
      k_2 &= \frac{\alpha_3 W}{2 \beta_2 W_1}, \label{eq: third-configuration-k2-02} \\
      k_3 &= \frac{\alpha_3(\beta_2 - \frac{W}{2W_1})}{\beta_2},  \label{eq: third-configuration-k3-02} \\
      b &= \frac{\alpha_3^2W_3}{\beta_2^2 W_1},  \label{eq: third-configuration-b-02} \\
      c &= \frac{\alpha_3^2}{W_1}.  \label{eq: third-configuration-c-02}
      \end{align}
  \item[c)] If $W > 0$ and $\alpha_0 + \alpha_3\beta_1 + \alpha_1\beta_3 - \alpha_2\beta_2 = 0$, then $Y(s)$ can be realized in the form of
      Fig.~\ref{fig: covering-configuration}(c), with
      \begin{align}
      k_1 &= \frac{\alpha_3 W_3 (\alpha_3 - \frac{W}{2W_3})}{\beta_3 W_2}, \label{eq: second-configuration-k1-02}  \\
      k_2 &= \frac{\alpha_3 W}{2 \beta_3 W_2},  \label{eq: second-configuration-k2-02} \\
      k_3 &= \frac{\alpha_3 (\beta_3 - \frac{W}{2W_2})}{\beta_3}, \label{eq: second-configuration-k3-02}   \\
      b &= \frac{\alpha_3^2}{W_2},  \label{eq: second-configuration-b-02} \\
      c &= \frac{\alpha_3^2W_3}{\beta_3^2 W_2}.  \label{eq: second-configuration-c-02}
      \end{align}
  \item[d)] If $W > 0$ and $\alpha_0 + \alpha_1\beta_3 + \alpha_2\beta_2 - \alpha_3\beta_1 = 0$, then $Y(s)$ can be realized in the form of
      Fig.~\ref{fig: covering-configuration}(d), with
      \begin{align}
      k_1 &= \frac{W_1 (\beta_2 - \frac{W}{2W_1})}{\beta_2 \beta_3},  \label{eq: fourth-configuration-k1-02}  \\
      k_2 &= \frac{W}{2\beta_2\beta_3},  \label{eq: fourth-configuration-k2-02} \\
      k_3 &= \frac{W_2(\beta_3 - \frac{W}{2W_2})}{\beta_2\beta_3}, \label{eq: fourth-configuration-k3-02} \\
      b &= \frac{W_2}{\beta_2^2},  \label{eq: fourth-configuration-b-02}  \\
      c &= \frac{W_1}{\beta_3^2}.  \label{eq: fourth-configuration-c-02}
      \end{align}
\end{enumerate}
\end{proof}

\section{Proof of Theorem~6}
\subsection{A Previous Lemma}
\begin{lemma}  \label{lemma: two-port}
{If the admittance of a two-port network consisting of only springs is well-defined, then it must be realizable as shown in Fig.~\ref{fig: two-port-configuration}, with $k_1$, $k_2$, $k_3$ $\geq 0$ or the one by switching the polarity of any one of the two ports.}
\end{lemma}
\begin{proof}
For a two-port network consisting of only springs, its admittance $Y_L$ must be in the form of
\begin{equation}  \label{eq: Y_L}
Y_L = \frac{1}{s} \left[
           \begin{array}{cc}
             K_{11} & K_{12} \\
             K_{12} & K_{22} \\
           \end{array}
         \right] := \frac{1}{s} K.
\end{equation}
According to the analogy to one-element-kind networks, $K$ as defined in \eqref{eq: Y_L} is necessarily paramount \cite{SW58}, that is, $K_{11} \geq |K_{12}|$ and  $K_{22} \geq |K_{12}|$.

If $K_{12} \geq 0$, then $Y_L$ as in \eqref{eq: Y_L} must be realizable as shown in Fig.~\ref{fig: two-port-configuration}, where $k_1 = K_{11} - K_{12} \geq 0$, $k_2 = K_{12} \geq 0$, and $k_3 = K_{22} - K_{12} \geq 0$.

If $K_{12} < 0$, then $Y_L$ as in \eqref{eq: Y_L} must be realizable as shown in Fig.~\ref{fig: two-port-configuration} by switching the polarity of any one of the two ports, where $k_1 = K_{11} - |K_{12}| \geq 0$, $k_2 = |K_{12}| > 0$, and $k_3 = K_{22} - |K_{12}| \geq 0$.
\end{proof}

\begin{figure}[thpb]
      \centering
      \includegraphics[scale=0.95]{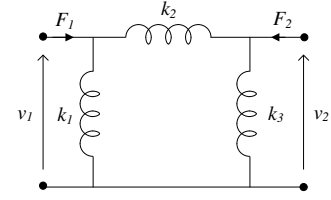}
      \caption{The two-port configuration consisting only of springs.}
      \label{fig: two-port-configuration}
\end{figure}

\subsection{Main Proof}

\textit{Theorem~6:}
A positive-real function $Y(s)$ can be realized as the
driving-point admittance of a one-port network,  consisting of one damper, one inerter, and at most three springs, and not satisfying Assumption~1, if and only if $Y(s)$ can be written in the form of
\begin{equation}  \label{eq: general admittance non-well-defined}
Y(s)=\frac{\alpha_3s^3+\alpha_2s^2+\alpha_1s+\alpha_0}{\beta_3s^3 + \beta_2s^2+\beta_1s},
\end{equation}
where $\alpha_0$, $\alpha_1$, $\alpha_2$, $\alpha_3$, $\beta_1$, $\beta_2$, $\beta_3$ $\geq 0$, and one of the following five conditions holds:
1)~$\alpha_3 = \beta_2 = \beta_3 = 0$, $\alpha_1$, $\alpha_2$, $\beta_1$ $> 0$; 2)~$\alpha_3 = 0$, $\beta_2$,
$\beta_3$ $> 0$, $\alpha_1\beta_1 - \alpha_0\beta_2 \geq 0$,
$\alpha_1^2 + \alpha_0\beta_2^2 \geq \alpha_1\beta_1\beta_2$,
$\alpha_0\beta_2 + \beta_2\beta_1^2 \geq \alpha_1\beta_1$,   $\alpha_1\beta_3 = \alpha_2\beta_2$;
3)~$\beta_2 = 0$, $\alpha_3$, $\beta_3$ $>0$, $\alpha_2\beta_1 - \alpha_0\beta_3 \geq 0$,
$\alpha_2^2 + \alpha_0\beta_3^2 \geq \alpha_2\beta_1\beta_3$,
$\alpha_0\beta_3 + \beta_3\beta_1^2 \geq \alpha_2\beta_1$,   $\alpha_1\beta_3 = \alpha_3\beta_1$;
4)~$\beta_3 = 0$, $\alpha_3$, $\beta_2$ $> 0$, $\alpha_1\beta_1 - \alpha_0\beta_2 \geq 0$,
$\alpha_1^2 + \alpha_0\beta_2^2 \geq \alpha_1\beta_1\beta_2$,
$\alpha_0\beta_2 + \beta_2\beta_1^2  \geq \alpha_1\beta_1$,   $\alpha_3\beta_1 = \alpha_2\beta_2$;
5)~$\alpha_3$, $\beta_2$, $\beta_3$ $> 0$, $\alpha_1\beta_3 + \alpha_2\beta_2 \geq \alpha_3\beta_1$,
$\alpha_2\beta_2 + \alpha_3\beta_1  \geq \alpha_1\beta_3$, $\alpha_1\beta_3 + \alpha_3\beta_1 \geq \alpha_2\beta_2$,
$\alpha_3 = \beta_2\beta_3$,
$\alpha_1^2\beta_3^2 + \alpha_2^2\beta_2^2 + \alpha_3^2\beta_1^2 + 4\alpha_0\alpha_3^2 = 2 (\alpha_1\beta_3\alpha_2\beta_2 + \alpha_2\beta_2\alpha_3\beta_1 + \alpha_3\beta_1\alpha_1\beta_3)$.

Furthermore, networks in Fig.~\ref{fig: configuration-non-well-defined} with $b$, $c$ $> 0$ and $k_1$, $k_2$, $k_3$ $\geq 0$ can  realize each of the five conditions above, respectively.
\begin{figure}[thpb]
      \centering
      \subfigure[]{
      \includegraphics[scale=0.87]{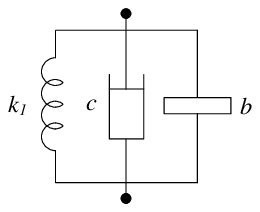}
      \label{subfig: first-configuration-non-well-defined}}
      \subfigure[]{
      \includegraphics[scale=0.87]{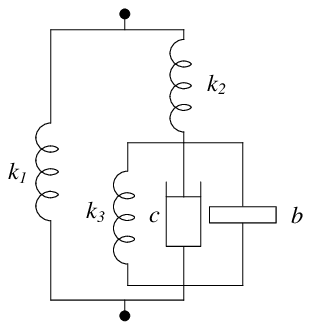}
      \label{subfig: second-configuration-non-well-defined}}
      \subfigure[]{
      \includegraphics[scale=0.87]{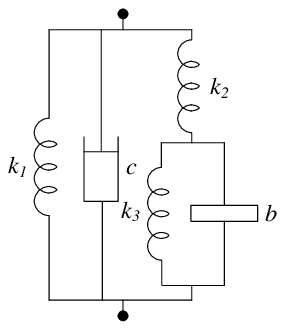}
      \label{subfig: third-configuration-non-well-defined}}
      \subfigure[]{
      \includegraphics[scale=0.87]{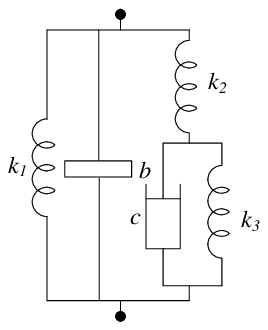}
      \label{subfig: fourth-configuration-non-well-defined}}
      \subfigure[]{
      \includegraphics[scale=0.87]{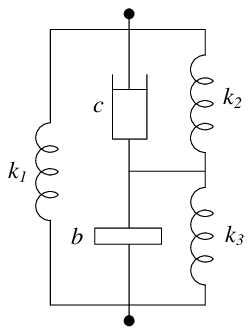}
      \label{subfig: fifth-configuration-non-well-defined}}
      \caption{The networks used to cover the condition of Theorem~6. (a) for Condition~1; (b) for Condition~2; (c) for Condition~3; (d) for Condition~4; (e) for Condition~5.}
      \label{fig: configuration-non-well-defined}
   \end{figure}
\begin{proof}
Based on Lemma~\ref{lemma: well-defined} and \cite[pg.~27]{SR61},  for the augmented graph $\mathcal{G}$, the edges corresponding to the inerter, the damper, and the port must form at least one circuit when the admittance of $X$ is not well-defined. Denote $e_i$ as the edge corresponding to the inerter, $e_d$ to the damper, and $e_p$ to the port. Then, all the possible cases are as follows.
If each pair of the three edges constitute a circuit, then the possible network must be equivalent to  Fig.~\ref{fig: configuration-non-well-defined}(a), where $b$, $c$ $> 0$ and $k_1 \geq 0$, because any other nodes can be eliminated by the generalized star-mesh transformation \cite{Ver70}.
If there is only one circuit constituted by two edges of $e_i$, $e_d$, and $e_p$, then the other part of the network can be regarded as a two-port network which consists of only springs and obtains a well-defined admittance. Therefore, based on Lemma~\ref{lemma: two-port}, one obtains the networks shown in Fig.~\ref{fig: configuration-non-well-defined}(b), \ref{fig: configuration-non-well-defined}(c) and \ref{fig: configuration-non-well-defined}(d), where $b$, $c$ $> 0$ and $k_1$, $k_2$, $k_3$  $\geq 0$.
If the three edges together form one circuit, then an equivalent configuration is shown in Fig.~\ref{fig: configuration-non-well-defined}(e), where $b$, $c$ $> 0$ and $k_1$, $k_2$, $k_3$  $\geq 0$, by using the generalized star-mesh transformation \cite{Ver70}.
Now, it remains to derive the realizability conditions for the networks in Fig.~\ref{fig: configuration-non-well-defined}, each of which corresponds to one of the five conditions of this theorem.

\textit{The realizability condition of Fig.~\ref{fig: configuration-non-well-defined}(a):} It is calculated that the admittance of the network shown in Fig.~\ref{fig: configuration-non-well-defined}(a) is expressed as
\begin{equation} \label{eq: first-configuration-non-well-defined}
Y(s) = \frac{b s^2 + c s + k_1}{s},
\end{equation}
where $b$, $c$ $> 0$ and $k_1 \geq 0$. If $Y(s)$ is realizable as in Fig.~\ref{fig: configuration-non-well-defined}(a), then it can also be expressed as \eqref{eq: general admittance non-well-defined}, with $\alpha_3 = 0$, $\alpha_2 = b$, $\alpha_1 = c$, $\alpha_0 = k_1$, $\beta_3 = 0$, $\beta_2 = 0$, and $\beta_1 = 1$. It is obvious that they are all non-negative and satisfy Condition~1. Conversely, if $\alpha_0$, $\alpha_1$, $\alpha_2$, $\alpha_3$, $\beta_1$, $\beta_2$, $\beta_3$ $\geq 0$ and they satisfy Condition~1, then letting $b = \alpha_2/\beta_1$, $c = \alpha_1/\beta_1$, and $k_1 = \alpha_0/\beta_1$ yields that $b$, $c$ $> 0$, $k_1 \geq 0$, and \eqref{eq: first-configuration-non-well-defined} equals \eqref{eq: general admittance non-well-defined}.

\textit{The realizability condition of Fig.~\ref{fig: configuration-non-well-defined}(b):}
It is calculated that the admittance of the network shown in Fig.~\ref{fig: configuration-non-well-defined}(b) is expressed as
\begin{equation} \label{eq: second-configuration-non-well-defined}
Y(s) = \frac{b(k_1+k_2)s^2 + c(k_1+k_2)s + (k_1k_2+k_2k_3+k_1k_3)}{s\left(  bs^2 + cs + (k_2+k_3) \right)},
\end{equation}
where $b$, $c$ $> 0$ and $k_1$, $k_2$, $k_3$ $\geq 0$. If $Y(s)$ is realizable as in Fig.~\ref{fig: configuration-non-well-defined}(b), then it can also be expressed as \eqref{eq: general admittance non-well-defined}, with $\alpha_3 = 0$, $\alpha_2 = b(k_1+k_2)$, $\alpha_1 = c(k_1+k_2)$, $\alpha_0 = k_1k_2 + k_2k_3 + k_1k_3$, $\beta_3 = b$, $\beta_2 = c$, and $\beta_1 = k_2+k_3$. It is obvious that they are all non-negative and satisfy Condition~2. Conversely, if $\alpha_0$, $\alpha_1$, $\alpha_2$, $\alpha_3$, $\beta_1$, $\beta_2$, $\beta_3$ $\geq 0$ and they satisfy Condition~2, then letting $b = \beta_3$, $c = \beta_2$, $k_1 = \alpha_1/\beta_2 - k_2$, $k_3 = \beta_1 - k_2$, and $k_2 = \sqrt{(\alpha_1\beta_1 - \alpha_0\beta_2)/\beta_2}$ yields that $b$, $c$ $> 0$, $k_1$, $k_2$, $k_3$ $\geq 0$, and \eqref{eq: second-configuration-non-well-defined} equals \eqref{eq: general admittance non-well-defined}.

\textit{The realizability condition of Fig.~\ref{fig: configuration-non-well-defined}(c):}
It is calculated that the admittance of the network shown in Fig.~\ref{fig: configuration-non-well-defined}(c) is expressed as
\begin{equation}  \label{eq: third-configuration-non-well-defined}
\begin{split}
&Y(s) =  \\
&\frac{bcs^3 + b(k_1+k_2)s^2 + c(k_2+k_3)s + (k_1k_2+k_2k_3+k_1k_3)}{s\left( bs^2 + (k_2+k_3) \right)},
\end{split}
\end{equation}
where $b$, $c$ $> 0$ and $k_1$, $k_2$, $k_3$ $\geq 0$.  If $Y(s)$ is realizable as in Fig.~\ref{fig: configuration-non-well-defined}(c), then it can also be expressed as \eqref{eq: general admittance non-well-defined}, with $\alpha_3 = bc$, $\alpha_2 = b(k_1 + k_2)$, $\alpha_1 = c(k_2+k_3)$, $\alpha_0 = k_1k_2 + k_2k_3 + k_1k_3$, $\beta_3 = b$, $\beta_2 = 0$, and $\beta_1 = k_2 + k_3$. It is obvious that they are all non-negative and satisfy Condition~3. Conversely, if $\alpha_0$, $\alpha_1$, $\alpha_2$, $\alpha_3$, $\beta_1$, $\beta_2$, $\beta_3$ $\geq 0$ and they satisfy Condition~3, then letting $b=\beta_3$, $c = \alpha_3/\beta_3$, $k_1 = \alpha_2/\beta_3 - k_2$, $k_3 = \beta_1 - k_2$, and $k_2 = \sqrt{(\alpha_2\beta_1 - \alpha_0\beta_3)/\beta_3}$ yields that $b$, $c$ $> 0$, $k_1$, $k_2$, $k_3$ $\geq 0$, and \eqref{eq: third-configuration-non-well-defined} equals \eqref{eq: general admittance non-well-defined}.

\textit{The realizability condition of Fig.~\ref{fig: configuration-non-well-defined}(d):}
It is calculated that the admittance of the network shown in Fig.~\ref{fig: configuration-non-well-defined}(d) is expressed as
\begin{equation}  \label{eq: fourth-configuration-non-well-defined}
\begin{split}
&Y(s) = \\
&\frac{bcs^3 + b(k_2+k_3)s^2 + c(k_1+k_2)s + (k_1k_2 + k_2k_3 + k_1k_3)}{s\left( cs + (k_2+k_3) \right)},
\end{split}
\end{equation}
where $b$, $c$ $> 0$ and $k_1$, $k_2$, $k_3$ $\geq 0$.  If $Y(s)$ is realizable as in Fig.~\ref{fig: configuration-non-well-defined}(d), then it can also be expressed as \eqref{eq: general admittance non-well-defined}, with $\alpha_3 = bc$, $\alpha_2 = b(k_2+k_3)$, $\alpha_1 = c(k_1+k_2)$, $\alpha_0 = k_1k_2 + k_2k_3 + k_1k_3$, $\beta_3 = 0$, $\beta_2 = c$, and $\beta_1 = k_2+k_3$. It is obvious that they are all non-negative and satisfy Condition~4. Conversely, if $\alpha_0$, $\alpha_1$, $\alpha_2$, $\alpha_3$, $\beta_1$, $\beta_2$, $\beta_3$ $\geq 0$ and they satisfy Condition~4, then letting $b = \alpha_3/\beta_2$, $c = \beta_2$, $k_1 = \alpha_1/\beta_2 - k_2$, $k_3 = \beta_1 - k_2$, and $k_2 = \sqrt{(\alpha_1\beta_1 - \alpha_0\beta_2)/\beta_2}$ yields that $b$, $c$ $> 0$, $k_1$, $k_2$, $k_3$ $\geq 0$, and \eqref{eq: fourth-configuration-non-well-defined} equals \eqref{eq: general admittance non-well-defined}.

\textit{The realizability condition of Fig.~\ref{fig: configuration-non-well-defined}(e):}
It is calculated that the admittance of the network shown in Fig.~\ref{fig: configuration-non-well-defined}(e) is expressed as
\begin{equation}   \label{eq: fifth-configuration-non-well-defined}
\begin{split}
&Y(s) = \\
&\frac{bcs^3 + b(k_1+k_2)s^2 + c(k_1+k_3)s + (k_1k_2+k_2k_3+k_1k_3)}{s\left( bs^2 + cs + (k_2+k_3) \right)},
\end{split}
\end{equation}
where $b$, $c$ $\geq 0$ and $k_1$, $k_2$, $k_3$ $\geq 0$.  If $Y(s)$ is realizable as in Fig.~\ref{fig: configuration-non-well-defined}(e), then it can also be expressed as \eqref{eq: general admittance non-well-defined} with $\alpha_3 = bc$, $\alpha_2 = b(k_1+k_2)$, $\alpha_1 = c(k_1+k_3)$, $\alpha_0 = k_1k_2 + k_2k_3 + k_1k_3$, $\beta_3 = b$, $\beta_2 = c$, and $\beta_1 = k_2+k_3$. It is obvious that they are all non-negative and satisfy Condition~5. Conversely, if $\alpha_0$, $\alpha_1$, $\alpha_2$, $\alpha_3$, $\beta_1$, $\beta_2$, $\beta_3$ $\geq 0$ and they satisfy Condition~5, then letting $b = \beta_3$, $c = \beta_2$, $k_1 = ( \alpha_1\beta_3 + \alpha_2\beta_2  - \alpha_3\beta_1)/(2\alpha_3)$, $k_2 = ( \alpha_2\beta_2 - \alpha_1\beta_3  + \alpha_3 \beta_1)/(2\alpha_3)$, and $k_3 = ( \alpha_1\beta_3 - \alpha_2\beta_2  + \alpha_3\beta_1)/(2\alpha_3)$ yields that $b$, $c$ $> 0$, $k_1$, $k_2$, $k_3$ $\geq 0$, and \eqref{eq: fifth-configuration-non-well-defined} equals \eqref{eq: general admittance non-well-defined}.
\end{proof}

\section{Realizability Conditions of One-Port Networks with one damper, one inerter, and an Arbitrary Number of Springs in Terms of \eqref{eq: general admittance final}}

In \cite{CS09}, a necessary and sufficient condition has been derived for any positive-real function to be realizable as the admittance of the one-port network consisting of one damper, one inerter, and an arbitrary number of springs as shown in \cite[Fig.~5]{CS09} where  $X$ has a well-defined impedance. If, instead of the impedance, it is the admittance of the three-port network $X$ that is assumed to be well-defined (Assumption~1), then the obtained necessary and sufficient condition for the realization of the one-port network consisting of one damper, one inerter, and an arbitrary number of springs becomes that $Y(s)$ can be written in the form of \eqref{eq: Y general}, where  $G$ is non-negative definite and satisfies the conditions of either \cite[Lemma~3]{CS09} or \cite[Lemma~4]{CS09} with $R$ being replaced by $G$, the detail of which is listed as follows.

\begin{theorem}  \label{theorem: equivalent condition}
{A positive-real function $Y(s)$ is realizable as the driving-point admittance of the one-port network,  consisting of one damper, one inerter, and an arbitrary number of springs, and satisfying Assumption~1, if and only if $Y(s)$ can be written in the form of \eqref{eq: Y general}, where $G$ as defined in \eqref{eq: G} is non-negative definite and satisfies the conditions of Lemma~2 or one of the following five conditions:
\begin{itemize}
  \item[1.] $G_4G_5G_6 < 0$;
  \item[2.] $G_4G_5G_6 > 0$, $G_1 > (G_4G_5/G_6)$, $G_2 > (G_4G_6/G_5)$, and $G_3 > (G_5G_6/G_4)$;
  \item[3.] $G_4G_5G_6 > 0$, $G_3 < (G_5G_6/G_4)$, and $G_1G_2G_3 + G_4G_5G_6 - G_1G_6^2 - G_2G_5^2 \geq 0$;
  \item[4.] $G_4G_5G_6 > 0$, $G_2 < (G_4G_6/G_5)$, and $G_1G_2G_3 + G_4G_5G_6 - G_1G_6^2 - G_3G_4^2 \geq 0$;
  \item[5.] $G_4G_5G_6 > 0$, $G_1 < (G_4G_5/G_6)$, and $G_1G_2G_3 + G_4G_5G_6 - G_3G_4^2 - G_2G_5^2 \geq 0$.
\end{itemize}
}
\end{theorem}

It can be checked that $G_1G_2G_3 + G_4G_5G_6 - G_1G_6^2 - G_2G_5^2 = 0$, $G_1G_2G_3 + G_4G_5G_6 - G_1G_6^2 - G_3G_4^2 = 0$, and $G_1G_2G_3 + G_4G_5G_6 - G_2G_5^2 - G_3G_4^2 = 0$ imply $G_3 < (G_5G_6/G_4)$, $G_2 < (G_4G_6/G_5)$, and $G_1 < (G_4G_5/G_6)$, respectively. Hence, the condition of Theorem~2 is a proper subset of that of Theorem~\ref{theorem: equivalent condition}.

Let $m_1 = G_6$, $m_1^\dag = G_5$, $m_2 = G_2 - G_4G_6/G_5$, $m_2^\dag = G_1 - G_4G_5/G_6$, and $m_3 = G_3 - G_5G_6/G_4$.
Let $\lambda_1 = G_1G_2G_3 + G_4G_5G_6 - G_1G_6^2 - G_2G_5^2$, $\lambda_2 = G_1G_2G_3 + G_4G_5G_6 - G_1G_6^2 - G_3G_4^2$,   $\lambda_3 = G_1G_2G_3 + G_4G_5G_6 - G_3G_4^2 - G_2G_5^2$, and $\lambda_4 = \det(G)$.
For the admittance as in \eqref{eq: Y general}, a $G_5$--$G_6$ graph is as in Fig.~\ref{fig: G5G6} by letting $G_1 = G_2 = G_3 = 1$ and $G_4 = 0.5$ to show   realizable sets  when Assumption~1 holds. The shaded region represents the realizability condition for the networks with one damper, one inerter, and an arbitrary number of springs; the boundaries of the shaded region ($\lambda_1$, $\lambda_2$, $\lambda_3$ $=0$, and $\lambda_4$ $= 0$ for $G_4G_5 < 0$), and the line segments inside the shaded region ($m_1$, $m_1^{\dag}$, $m_2$, $m_2^{\dag}$, $m_3$ $= 0$) constitute the realizability condition for networks with one damper, one inerter, and at most three springs.

\begin{figure}[thpb]
      \centering
      \includegraphics[scale=0.6]{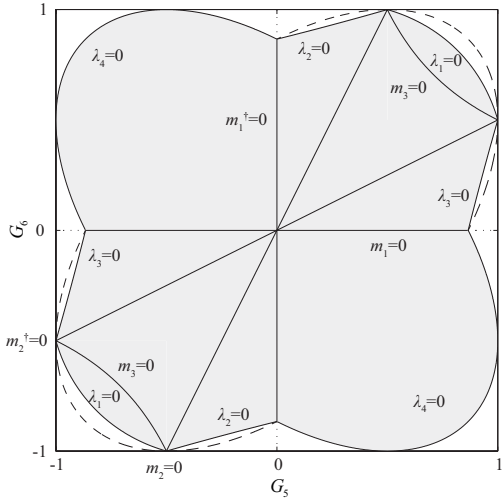}
      \caption{The $G_5$--$G_6$ graph showing realizable sets  assuming Assumption~1 holds, where $G_1 = G_2 = G_3 = 1$ and $G_4 = 0.5$.}
      \label{fig: G5G6}
\end{figure}

Using the coefficient transforation, one can further derive the following conclusion.

\begin{theorem}  \label{theorem: further equivalent condition}
{A positive-real function $Y(s)$ can be realized as the
driving-point admittance of a one-port network, consisting of one damper, one inerter, and an arbitrary number of springs, and satisfying Assumption~1, if and only if $Y(s)$ can be written in the form of \eqref{eq: Y initial},
where the coefficients satisfy $\alpha_0$, $\alpha_1$, $\alpha_2$, $\alpha_3$, $\beta_1$, $\beta_2$, $\beta_3$ $\geq 0$, $W_1$, $W_2$, $W_3$ $\geq 0$, $W^2 = 4W_1W_2W_3$ and also satisfy the condition
of Lemma~\ref{lemma: condition 01} or one of the following five conditions:
\begin{enumerate}
  \item[1.] $W < 0$;
  \item[2.] $W > 0$, $\alpha_3 - W/(2W_3) > 0$, $\beta_3 - W/(2W_2)$, and $\beta_2 - W/(2W_1) > 0$;
  \item[3.] $W > 0$, $\beta_2 - W/(2W_1) < 0$, and $\alpha_0 + \alpha_3\beta_1 + \alpha_1\beta_3 - \alpha_2\beta_2 \geq 0$;
  \item[4.] $W > 0$, $\beta_3 - W/(2W_2) < 0$, and $\alpha_0 + \alpha_3\beta_1 + \alpha_2\beta_2 - \alpha_1\beta_3 \geq 0$;
  \item[5.] $W > 0$, $\alpha_3 - W/(2W_3) < 0$, and $\alpha_0 + \alpha_1\beta_3 + \alpha_2\beta_2 - \alpha_3\beta_1 \geq 0$.
\end{enumerate}
}
\end{theorem}

Since $\beta_2 - W/(2W_1) = -\alpha_0/W_1$ when $\alpha_0 + \alpha_3\beta_1 + \alpha_1\beta_3 - \alpha_2\beta_2 = 0$, $\beta_3 - W/(2W_2) = -\alpha_0/W_2$ when $\alpha_0 + \alpha_3\beta_1 + \alpha_2\beta_2 - \alpha_1\beta_3 = 0$, and $\alpha_3 - W/(2W_3) = -\alpha_0/W_3$ when $\alpha_0 + \alpha_1\beta_3 + \alpha_2\beta_2 - \alpha_3\beta_1 = 0$, the condition of
Theorem~5 is a proper subset of that of Theorem~\ref{theorem: further equivalent condition}.

From the proof of Theorem~6, one can also see that if $Y(s)$ is realizable as a one-port network with one damper, one inerter, and an arbitrary number of springs, where $X$ does not have a
well-defined admittance, then the number of springs can be at most three. Hence, the result is obtained as follows.
\begin{theorem}    \label{theorem: equivalent final result}
{A positive-real function $Y(s)$ can be realized as the
driving-point admittance of a one-port network consisting of one inerter, one damper, and  an arbitrary number of springs, if and only if $Y(s)$ can be written in the form of
\begin{equation}  \label{eq: general admittance final}
Y(s)=\frac{\alpha_3s^3+\alpha_2s^2+\alpha_1s+\alpha_0}{\beta_4s^4 + \beta_3s^3 + \beta_2s^2+\beta_1s},
\end{equation}
where $\alpha_0$, $\alpha_1$, $\alpha_2$, $\alpha_3$, $\beta_1$, $\beta_2$, $\beta_3$ $\geq 0$, and they satisfy the condition of Theorem~\ref{theorem: further equivalent condition} when $\beta_4 = 1$ or the condition of Theorem~6 when $\beta_4 = 0$. }
\end{theorem}


%


\end{document}